\documentclass{amsart}  

\usepackage[utf8]{inputenc} 
\usepackage{tikz}
\usepackage{array} 
\usepackage[colorlinks,linkcolor=blue]{hyperref} 
\usepackage{doi}

\renewcommand{\P}{\mathbb{P}}
\newcommand{\Z}{\mathbb{Z}}
\newcommand{\N}{\mathbb{N}}
\newcommand{\A}{\mathbb{A}}

\newcommand{\G}{\mathbb{G}}
\newcommand{\vp}{\varphi}

\newcommand {\m}{\mathfrak{m}}

\let \pargr \S
\renewcommand{\S}{\mathcal S}
\renewcommand{\L}{\mathrm L}
\renewcommand{\k}{\mathbf{k}} 
\DeclareMathOperator{\Spec}{Spec}

\DeclareMathOperator{\Coker}{Coker}
\DeclareMathOperator{\Hom}{Hom}
\DeclareMathOperator{\End}{End}
\DeclareMathOperator{\Ann}{Ann}
\DeclareMathOperator{\res}{res}
\renewcommand{\ell}{l}

\newcommand{\mO}{\mathcal{O}}
\makeatletter
\newcommand*{\eqdef}{=\mathrel{\rlap{%
                     \raisebox{0.3ex}{$\m@th\cdot$}}%
                     \raisebox{-0.3ex}{$\m@th\cdot$}}}
\makeatother

\def\y(#1){y_{#1}}
\def\x(#1){x_{#1}}
\def\t(#1){s_{#1}}

\newtheorem{thm}{Theorem}[section]

\newtheorem{lemma}[thm]{Lemma}
\newtheorem{cor}[thm]{Corollary}
\newtheorem{prop}[thm]{Proposition}
\theoremstyle{definition}
\newtheorem{defn}[thm]{Definition}
\newtheorem{ex}[thm]{Example}
\newtheorem{conj}[thm]{Conjecture}
\theoremstyle{remark}
\newtheorem{remark}[thm]{Remark}

\begin{document}
\title{Non-smoothable curve singularities}

\author{Jan Stevens}
\address{\scriptsize  Department of Mathematical Sciences,  Chalmers University of
Technology and University of Gothenburg.
SE 412 96 Gothenburg, Sweden}
\email{stevens@chalmers.se}

\begin{abstract}
For curves  singularities  all smoothing components  of  the deformation space
have the same dimension, but there can be components of different dimensions.               
We are interested in the question of what the generic singularities are
that appear in the fibre over a component.  
To this end we revisit the known examples of non-smoothable
singularities and study their deformations. 

There are two general methods available to show that a curve is not smoothable.
In the first method one exhibits a family of singularities of a certain type
and then uses a dimension count to prove that the family cannot lie in the
closure of the space of smooth curves. The other method  uses the semicontinuity of a certain invariant, related to the Dedekind
different. This invariant vanishes for Gorenstein singularities, so in particular for
smooth curves. With these methods and computations with computer algebra systems 
we study monomial curves and cones over point sets in projective space.

We also give new explicit examples
of non-smoothable singularities. In particular, we find non-smoothable
Gorenstein curve singularities. The cone over a general self-associated point set
in $\mathbb P^{g-2}$ is not smoothable if $g$ 
is at least 11, as then the point set can
not be a hyperplane section of a canonical curve of genus $g$.
\end{abstract}

\thanks{}
\maketitle

\section*{Introduction}
Deformation spaces of singularities can be very singular \cite{Va}. 
Reduced curve singularities seem to be better behaved, as the dimension of a smoothing component
of the versal deformation, that is a component over which the general fibre is smooth,
is an invariant of the singularity \cite{De}.
But  if the  singularity deforms
into a non-smoothable curve singularity,
there are in general also components of other dimensions.

One of the purposes of this paper is to investigate which singularities can occur as
general fibre over a (non smoothing) component. Such singularities may be called
generic singularities. To this end we revisit the known examples of non-smoothable
singularities and study their deformations. We also give new explicit examples
of non-smoothable singularities. In particular, we find non-smoothable
Gorenstein curve singularities.

The existence of non-smoothable curves was first shown by Mumford \cite{Mu}. 
He constructed families of curves with one irreducible singularity, 
such that the general element of the family is not
smoothable, basically because the family is too large to come from a closure
of  the moduli space of smooth curves. New examples, of lines through the origin, were
given by Pinkham \cite{Pi}.  These examples were treated and extended  by Greuel \cite{Gr0, Gr}, 
see also the survey \cite{Gr2}, using a local argument. The dimension of  a 
smoothing component can be expressed in terms of invariants of the curve singularity $C$,
see Proposition \ref{PropDeligne}; we call this number the \textit{Deligne number} $e(C)$. The idea is now to exhibit a large family
of singularities $\pi \colon C_T \to T$ with singular section,  such that for $t\neq0$
the singularity $(C_t,\sigma(t))$ is not isomorphic to $(C_0,\sigma(0))$. Here large means
that the dimension of $T$ is at least the Deligne number $e(C_0)$. If $(T,0)$ is irreducible, 
then the general $C_t$ is not smoothable, because the image of $T$ in the versal deformation has dimension $\dim T \geq e(C_0)$, but cannot be a smoothing component, as there are no smooth fibres over $T$.
For quasi-homogeneous curves the number $e$ can be expressed in familiar
invariants of the curve, making this an effective criterion.

This \textit{large  family argument} shows the existence of non-smoothable singularities, but does not 
give specific examples. The first such examples of monomial curves are due to Buchweitz \cite{Bu}.
He observed that a necessary condition for smoothability, in fact for deforming into
a Gorenstein singularity, is that the length of the Dedekind different is at most $2\delta$, something
conjectured,  or at least seen as a possibility,  to hold for all singularities by Herzog \cite{He}.
More generally, Buchweitz defined a $k$-th normalised Dedekind invariant. Furthermore he showed 
how to compute this invariant for monomial curves in terms of the semigroup. 
This has applications to Weierstrass points. A semigroup is the semigroup of
a Weierstrass point on a smooth projective curve, if the curve occurs as a deformation
of the completion of the corresponding monomial curve in an appropriate weighted
projective space, with the Weierstrass point at infinity. The non-smoothability
condition is thus also a condition that the semigroup cannot be a
Weierstrass semigroup, and this application can be proven directly from Riemann-Roch. 
In this form
Buchweitz' criterion occurs most often in the literature, but the original criterion
applies more generally.

These two methods  to prove that singularities are not smoothable
are the only known general methods. In certain cases one method succeeds, in other
cases the other. For not too complicated singularities direct computation of infinitesimal
deformations with computer algebra systems, in particular Singular \cite{Sing} and Macaulay2 \cite{M2},
can also be used in studying smoothability. Experimentation  with such singularities has led to conjectures, which can be found throughout this paper. Not too complicated means in
practice quasi-homogeneous. The $\mathbb G_m$-actions vastly simplifies computations.
Also the occurring invariants of singularities are easier to compute in the quasi-homogeneous
case. For monomial curves the generators of the defining ideal are particular simple, as they 
are binomials.  We do give examples of non quasi-homogeneous non-smoothable curve
singularities, but they occur as deformations of non-smoothable quasi-homogeneous ones.
If a singularity is not smoothable, then by openness of versality every singularity into which
it deforms, is also not smoothable. 

The Buchweitz criterion does not give any non-trivial condition for Gorenstein curves and therefore it cannot be used 
to show that symmetric semigroups are not Weierstrass semigroups. A double cover
construction   yields symmetric non-Weierstrass semigroups \cite{To}. We study
deformations of the simplest example and show that it deforms into a Gorenstein curve
consisting of lines through the origin, which is not smoothable. This singularity is the cone
over a self-associated point set, a concept introduced by Coble \cite{Co}.
The condition that the general cone is smoothable is that the general self-associated point set is
a hyperplane section of a canonical curve.   
The large family argument shows that this is
in general not the case if the genus of the curves is at least $11$.

Besides showing that the general singularity of a certain type is not smoothable, we
give explicit examples of non-smoothable curve singularities. We mention in particular
a non-smoothable Gorenstein curve singularity (Proposition \ref{prop:goren}),
an irreducible smoothable but not negatively smoothable quasi-homogeneous curve 
(Example \ref{ex:negsmooth})
and a singularity where the dimension of the base space is less than the Deligne number $e$
(Example \ref{ex:L106}).

The evidence collected in this paper leads to the conjecture that
generic singularities have only smooth branches. The tangents to these branches are not necessarily
in general position.

\section{Preliminaries}

\subsection{}
Let $C$ be a reduced  affine curve over an algebraically closed field $\k$ of characteristic 0,
lying in $\A^n$,
and let $0\in C$ be a closed point. We denote by $\mO=\mO_{C,0}$ its local ring with 
maximal ideal $\m$. Let $n\colon (\overline C,n^{-1}(0))\to (C,0)$ be the normalisation
with semi-local ring $\overline{\mO}=n_*\mO_{\overline C,n^{-1}(0)}$. It is the 
integral closure of $\mO$ in its total ring of fractions $K$. Then the $\delta$-invariant of
$(C,0)$, sometimes called the degree of singularity or (mostly for plane curves) the 
number of virtual double points, is $\delta=\delta(C,0)=\dim_k \overline\mO/\mO$.
For reducible curves $C_1\cup C_2$ the $\delta$-invariant can be computed as
\[
\delta(C_1\cup C_2) = \delta(C_1) +  \delta(C_2)  + (C_1\cdot C_2)
\]
where the intersection multiplicity $ (C_1\cdot C_2)$ is given by 
 $ (C_1\cdot C_2)= \dim_k \mO_n/(I_1+I_2)$
with $I_1$ and $I_2$ the ideals of $C_1$ and $C_2$ in the local ring $\mO_n$ of $(\A^n,0)$.

The conductor ideal is $\mathcal C= \Ann_{\overline{\mO}}(\overline\mO/\mO) 
=\{f\in \overline{\mO}\mid f\overline{\mO} \subset \mO\}$ and 
 $c = \dim_k \overline\mO/\mathcal C$ is its multiplicity.

Let $\omega=\omega_{C,0}$ be the dualising module of $(C,0)$. It can be described
by Rosenlicht differentials: if $\Omega$ is the module of Kähler differentials on $(C,0)$
and $\overline \Omega=n_* \Omega_{\overline C,n^{-1}(0)}$ the module of
differentials on the normalisation, then $\omega=\{ \alpha \in \overline\Omega\otimes K
\mid \sum_{P\in n^{-1}(0)}\res_P (\alpha)=0\}$ \cite[IV.9]{Se}.
Composing the exterior derivation $d\colon \mO \to \Omega$ with the map
$\Omega\to\overline\Omega\hookrightarrow\omega$ gives a map
$d\colon \mO\to\omega$, which allows to define the Milnor number as
$\mu = \dim_k \omega/d\mO$ \cite{BG}. It satisfies Milnor's formula $\mu = 2\delta - r +1$,
where $r$ is the number of branches. We say that the genus of the curve singularity
is $g=\delta - r +1$. Over the complex numbers the genus is the genus of the Milnor fibre
in a smoothing and the first Betti number of the Milnor fibre is   $\mu = 2\delta - r +1$.

Let $(C,0)$ be embedded in $\A^n$ and write $\mO_n$ for the local ring of $\A^n$ at the
origin. Let $I=\langle f_1,\dots,f_k\rangle$ be the ideal of $(C,0)$. Then there is a 
minimal free resolution
\[
0\longleftarrow \mO \longleftarrow \mO_n   \longleftarrow \mO_n ^k \longleftarrow \mO_n^l
 \longleftarrow \dots \longleftarrow \mO_n^t \longleftarrow 0
 \]
 of length $n-1$. The rank of the last free module is the Cohen-Macaulay type $t$ of
 the curve. It is also the number of generators, $\dim_k \omega/\m \omega$,  of the
 dualising module. In fact, the dual of the free resolution gives a resolution of the
 dualising module.

\subsection{}
A deformation of $(C,0)$ over $(S,0)$  consists of a flat morphism $\pi\colon (C_S,0) \to (S,0)$
together with an isomorphism of $(C,0)$ with the special fibre $(C_0,0)$ of $\pi$.
Often we identify  $(C,0)$ with $(C_0,0)$ . Flatness can be characterised by the property 
that every free resolution of $\mO$ lifts to a free resolution of $\mO_{C_S,0}$ over the
local ring of $\A^n\times S$. It suffices that every relation $\sum f_i r_i$ between the
generators lifts to a relation $\sum F_iR_i$ between the generators  
$\langle F_1,\dots,F_k\rangle$ of the ideal of $(C_S,0)$. 

There exists a formally
versal formal deformation $C_B$ with $B$ the spectrum of a complete local $\k$-algebra with  $\k$
as residue field, see e.g., \cite{Ar}. 
A component  $E$ of the deformation space $B$ is called a smoothing component
if ``the generic fibre is smooth'',  that is, if the image of the formal scheme giving the singular locus
does not contain the generic point of $E$.
For curve singularities Deligne has given a formula for the dimension of smoothing components 
\cite[Thm. 2.27]{De}.

Let $\Theta= \Hom_\mO(\Omega,\mO)$ be the module of derivations on $(C,0)$
and  let $\overline\Theta= \Hom_{\overline\mO}(\overline\Omega,\overline\mO)$ be the
module of derivations on $(\overline C,n^{-1}(0))$. Then define
$m_1=\dim_k \overline\Theta/\Theta$.  

\begin{prop}[Deligne]\label{PropDeligne}
Every smoothing component $E$ of $(C,0)$ has dimension equal to the 
Deligne number  $e=3\delta -m_1$.
\end{prop}

The formula also holds in finite characteristic, but as then  not every derivation
of $\mO$ lifts to a derivation of $\overline\mO$, the definition
of $m_1$ has to be modified.  The derivations of $\mO$ and $\overline\mO$
lift to derivations of the total ring of fractions $K$ and under these embeddings 
$ \Theta$ and $ \overline\Theta$ are commensurable, allowing to define 
$m_1=\dim_k \overline\Theta/(\Theta\cap  \overline\Theta)-
\dim_k  \Theta/(\Theta\cap  \overline\Theta)$.

\subsection{}
Many computations simplify considerably if the singularity is quasi-homoge\-neous. In that case
one can work degree by degree. In computer algebra systems computations with 
weighted homogeneous equations are much faster.

In the quasi-homogeneous case the
 multiplicative  group  $\G_m$ of units of $\k$ acts diagonally, for suitably chosen  
  coordinates $(x_1,\dots,x_n)$ on $\A^n$, 
 by $(g,x_i)\mapsto g^{w_i}x_i$; the integers $w_i$ are called weights.
The coordinate ring 
$P/I$, where $P=\k[x_1,\dots,x_n]$, is a graded ring.  We continue to denote it by $\mO$.
The generators of the ideal $I=\langle f_1,\dots,f_k\rangle$ can be chosen to be
equivariant, with $f_j$ of degree $q_j$. A quasi-homogeneous
deformation can be described by power series $F_j$  of the same degree $q_j$, lifting the $f_j$,
if one assigns appropriate weights to the deformation variables.
If all the deformation variables have positive weights, then the $F_j$ are polynomials.
In that case the deformation describes a deformation of the closure of the curve
in a weighted projective space, which is trivial at infinity.
If we also allow weight zero, then the deformation still globalises.
A deformation with deformation variables of positive weight is said to be a deformation
of negative weight, as the degree of the terms in the space variables $x_i$ decreases.
This is in particular true for trivial deformations, those induced by coordinate
transformations. 
We note that 
 Pinkham \cite{Pi,Pi2} uses 
the Bourbaki definition of positive and negative, according to which 0 is both positive and negative. 
It is therefore customary to call a weighted homogeneous curve negatively smoothable if
there exists a smoothing that globalises to a smoothing of the closure of the curve, in the same 
weighted projective space.

In the quasi-homogeneous case, the formula for the dimension of smoothing components simplifies. 
\begin{prop}[Greuel \cite{Gr}]
For a quasi-homogeneous curve $(C,0)$ the 
Deligne number  $e$ is $e=\mu + t-1=2\delta-r+t$.
\end{prop}

Most known examples of non-smoothable curves fall under one of two extremes, either
they are irreducible, or have as many branches as possible, each of them smooth.

\subsubsection{Monomial curves}
Let $\S=\langle a_1,\dots,a_n\rangle$ be a numerical semigroup  
(for the general theory of numerical semigroups, see the book \cite{RG}).
The semigroup ring
$\k[\S]$ can be identified with the ring $\oplus_{a\in S} \k t^a$, a ring generated by monomials.
The curve $C_\S=\Spec \k[\S]$ is the associated monomial curve.
The complement $\N\setminus \S=:\L$ is the set of gaps. The number of gaps is called the
genus of the semigroup. It is equal to $g=\delta$ of the associated monomial curve.
The largest gap is called the Frobenius number $F(\S)$ and $c=F(\S)+1$ is the conductor:
the element $t^c$ generates the conductor ideal of $\k[\S]$.

The dualising module $\omega$ is easy to describe. It contains $\overline \Omega$
and $\omega/\overline\Omega$ is the vector space generated by the differentials
$\frac {dt}{t^{\ell+1}}$, $\ell\in\L$. The type of the curve is the number of generators 
of $\omega$, which is the number of gaps $\ell$  such that $\ell+n\in \S$ whenever
whenever $n\in \S\setminus 0$.

\subsubsection{Lines though the origin}\label{sssc:lines}
A special case of the
other extreme concerns homogeneous singularities. In this case the curve
is the cone over a set $\Gamma$ of points in $\P^{n-1}$, and consists of a set of lines
through the origin in  $\A^n$. We use the notation $L_r^n$ for any homogeneous  curve 
consisting of $r$ lines in  $\A^n$.  We are mostly interested in lines in general
position; 
for the definition  we follow the definitions in
\cite{Gr} and \cite{GO}.
\begin{defn}
A set $\Gamma$ of $r$ points in $\P^{n-1}$ is said to be in \textit{general position}, if
it imposes the  maximal number of conditions on forms of degree $d$, for all $d\geq 1$.
The points are in \textit{uniform position} if  every subset is in general position. 
A set of $r$ lines through the origin in $\A^n$ is in general  (uniform) position
if  the corresponding point set is in  general (uniform) position.
\end{defn}

The conditions can also be formulated  with the $d$-tuple Veronese embedding
$v_d\colon \P^{n-1} \to \P^{N-1}$, where $N=\binom{n+d-1}d$.
The point set $\{p_1,\dots,p_r\}$ is in general position, if for all $d$ the points 
$v_d(p_1),\dots,v_d(p_r)$ span a  linear subspace of dimension $\min(r,N)-1$.
After choosing homogeneous coordinates  for the points and an  ordering of the monomials
of degree $d$ we can write the matrix $A_d$ of homogeneous coordinates of the image
of the points under the  $d$-tuple Veronese embedding. General position just means that the
matrix $A_d$ has maximal rank, namely rank  $\min(r,N)$.

As the dimension of the space of forms of degree $d$ on $\P^{n-1}$ is $\binom{n+d-1}d=
\binom{n+d-1}{n-1}$, the Hilbert function of the homogeneous coordinate
ring $\mO$ of $\Gamma$ in general position is
\[
H(\ell) = \min\left(r,\binom{n+\ell-1}{n-1}\right)
\]
If $r\geq\binom{n+\ell-1}{n-1}$, there are no forms of degree $\ell$
in the homogeneous ideal of the points. 
We can construct a uniform set $\Gamma_r $ of $r$ points by adding one point at a time; note 
that this does not hold for general position. This is a practical way to do experiments with (cones over)
``random'' point sets.

We compute $\delta$ for a curve $L_r^n$ of $r$ lines through the origin in $\A^n$
in uniform position. We add a line  $L$ to a  $L_{r-1}^n$ in uniform position. 
Let $d$ be the lowest degree of a form vanishing on $L_{r-1}^n$. Then $L_{r}^n$
imposes independent conditions on forms of degree $d$, if there exist a form in the ideal of
$L_{r-1}^n$ that does not vanish on $L$. This gives that the intersection 
multiplicity $(L\cdot L_{r-1}^n)=d$. Therefore $\delta$ grows with $d$.
Using the formula for $H(\ell)$ we see that the lowest degree $d$ 
of a form vanishing on $L_{r-1}^n$ is determined by 
the conditions $\binom{n+d-2}{d-1} <r \leq \binom{n+d-1}{d}$. We introduce
a notation for this number.

\begin{defn}\label{defn:dnr}
Given two integers $n<r$, let $d(n,r)$ be the unique integer $d$ such
that $\binom{n+d-2}{d-1} <r \leq \binom{n+d-1}{d}$.
\end{defn}

With this notation,  the discussion above gives the following result.

\begin{lemma}\label{lem:delta}
Let  $L_r^n$ be a curves of lines in uniform position and set $d=d(n,r)$. Then
\[
\delta(L_r^n) = dr-\binom{n+d-1}{d-1}
\]
\end{lemma}

This result holds under the weaker condition of general position \cite[Lemma 3.3]{Gr}.
It can also be proved
using  the generalisation to higher embedding dimension of Noether's formula for the 
$\delta$-invariant \cite{St3}.

We describe the homogeneous parts of the dualising module $\omega$.
 As the conductor ideal $\mathcal C $ is equal to $\m^d$, where $d=d(n,r)$,
 the lowest degree part is $\omega_{-d}$.
The condition that  a rational differential form $\alpha =\sum a_i \frac{dt_i}{t_i^\ell} 
 \in \overline \Omega\otimes K_{-\ell}$ lies in $\omega_{-\ell}$ is that $\sum \res m \alpha=0$
 for each monomial of degree $\ell -1$. This gives a system of homogeneous linear equations on the
 coefficients of $\alpha$, whose matrix is the $\binom{n+\ell-2}{\ell-1}\times r$ matrix $A_\ell$
 with entries
the coordinates of the $r$ points under the Veronese embedding $v_{\ell-1}$.
 This matrix has maximal rank if the points are in general position.
 Therefore the dimension of $\omega_{-\ell}$ is $r-\binom{n+\ell-2}{\ell-1}$.
 The minimal number of generators for $\omega$ as $\mO$-module
is obtained  when multiplication with $\mO_1$ generates a subspace of $\omega_{-d+1}$
of largest possible dimension. The resulting formula for the type of the singularity
was conjectured by Roberts \cite{Ro} and the conjecture was proved by Trung and Valla \cite{TV} and Lauze \cite{La}, to hold for at least  one curve. Therefore it holds for 
a  generic $L_r^n$ in uniform position; in fact a generic  $L_r^n$ is  in uniform position.
Therefore we have:
 
 \begin{prop}\label{prop:type}
 For generic $L_n^r$ 
 write $r=\binom{n+d-2}{d-1} +s$, where  $d=d(n,r)$.
The type $t$ is $\max \{ s, \binom{n+d-3}{d-1}-(n-2)s\}$.
 \end{prop}

In fact it had been conjectured that all Betti numbers of the minimal free resolution
have the minimal possible value given the Hilbert function. For systematic counterexamples
see \cite{EP2}.

\subsubsection{Non-smoothable $L_r^n$}
By a result of Greuel \cite[3.5]{Gr} (see also Corollary \ref{cor:Tplusnul}), 
a curve $L_r^n$ of lines in general position  has no non-trivial
deformations of positive degree if $n<r\leq \binom{n+1}2$,
so it is smoothable if and only if the points lie in the closure of the space of
hyperplane sections of (non-special) curves of genus $g=r-n$ in $\P^n$. One
calls an embedded curve non-special if the linear system of hyperplane sections is 
non-special. For being a hyperplane section  there is a 
criterion in terms of the Gale transform of the points, that is in terms of associated point sets.

\begin{defn}
Let $\Gamma=\{p_1,\dots,p_r\}$ be a set of $r$ ordered points in $\P^{n-1}$,
whose homogeneous coordinates are given by a $n\times r$-matrix $P$. 
The Gale transform of $\Gamma$ is a set of $r$ points $\{q_1,\dots,q_r\}$
in $\P^{r-n-1}$ , given by a $(r-n) \times r$ matrix $ Q$ satisfying  $PQ^t=0$.
The two point sets are said to be associated. 
\end{defn}

This concept was introduced by Coble \cite{Co}.
For a detailed study, including an extensive historical overview, we refer to 
 \cite{EP}. If no $r-1$ points lie in a hyperplane we can normalise the matrix $P$ in the form 
 $(I,A)$ with $I$ the $n\times n$ identity matrix and no row of $A$ is zero. Then a solution 
 of  $PQ^t=0$ is $Q=(-A^t,I)$.
 
 We now have the following criterion \cite[Thm. 9]{St}, which is a special
 case of  \cite[Cor. 3.2]{EP}. 

\begin{thm}\label{thm:smooth-gale}
Let $\Gamma=\{p_1,\dots,p_r\}$ be a hyperplane section of a non-special curve
$C$ of genus $g$ in $\P^{r-g}$. Then the Gale transform  of $\Gamma$ lies on the
canonical image of $C$ in $\P^{g-1}$.
\end{thm}

This result enables us to construct simple explicit examples of non-smoothable curves.
The lowest possible embedding dimension is obtained for genus $g=4$.  
The general canonical curve is the complete intersection of a quadric and a cubic,
so for smoothability of the cone over a point set the points in the Gale transform have to lie on a quadric. The minimal number of points $r$
for which this is in general not the case is $10$. Then $n=r - g=6$.
Choosing a set of ten points on a quadric leads to a smoothable $L_{10}^6$ that deforms
into a non-smoothable one and therefore has a reducible base space.

\begin{ex}\label{ex:L106}
As a set of ten points that do not lie on a quadric we take  the four  vertices
of the coordinate tetrahedron and the six midpoints of the edges.

After taking a suitable basis the 
 11 equations of the Gale transform in $\P^5$ are
$x_1x_4=x_2x_5=x_3x_6=0$, 
$x_1x_2=x_1x_3=x_2x_3$,
$x_3x_4=x_3x_5=x_4x_5$,
$x_5x_6=x_5x_1=x_6x_1$ and 
$x_2x_4=x_2x_6=x_4x_6$. 
These points are not in uniform position, as each coordinate hyperplane contains seven points.
The advantage is that the equations are very simple.
A computation with Macaulay2 \cite{M2} or Singular \cite{Sing}
shows that $T^1$ of the corresponding
 $L_{10}^6$ is concentrated in
degree 0 and has dimension $15$, while $T^2=0$.
This shows that the versal deformation is smooth, and that this particular curve
has only equisingular deformations. The Deligne-Greuel formula for the dimension of a 
smoothing component gives $e=20$. Therefore we have here an example where the dimension  of the base space
of the versal deformation  is less than the Deligne number $e$.

\[
\begin{tikzpicture}
\pgfmathsetmacro{\factor}{1/sqrt(2)};
\coordinate  (A) at (2,0,-2*\factor);
\coordinate (B) at (-2,0,-2*\factor);
\coordinate  (C) at (0,2,2*\factor);
\coordinate  (D) at (0,-2,2*\factor);
\draw (A)--(D)--(B)--(A) --(C)--(B) (D)--(C);
\foreach \i in {A,B,C,D}
\draw[blue] [fill=blue] (\i) circle(0.1) ;
\draw[blue] [fill=blue] (0,0,-2*\factor) circle(0.1) ;
\draw[blue] [fill=blue] (1,1,0) circle(0.1) ;
\draw[blue] [fill=blue] (1,-1,0) circle(0.1) ;
\draw[blue] [fill=blue] (-1,1,0) circle(0.1) ;
\draw[blue] [fill=blue] (-1,-1,0) circle(0.1) ;
\draw[blue] [fill=blue] (0,0,2*\factor) circle(0.1) ;
\end{tikzpicture}
\]

A smoothable $L_{10}^6$ is obtained by replacing two opposite edge midpoints 
of the coordinate tetrahedron in $\P^3$ by the center $(1:1:1:1)$ and a point on the quadric 
through the centre and the remaining eight
points, say the point $(1:2:2:4)$ on the quadric $y_1y_4=y_2y_3$.
 The resulting  $L_{10}^6$  has $\dim T^1=21$, with $\dim T^1_0=15$   and $\dim T^1_{-1}=6$, while 
$\dim T^2 =\dim T^2_{-1} = 6$. The versal deformation has two smooth components, one smoothing
component of dimension 20, and the 15-dimensional component of equisingular deformations, intersecting
in a 14-dimensional space.
\end{ex}

\subsection{} 
To compute deformations the first step is to compute the vector space of
 infinitesimal deformations and the obstruction space, that is $T^1$ and $T^2$.
 In general  these vector spaces are part of  the theory of cotangent cohomology, 
 whose main properties relevant for the case at hand
 are summarised in \cite{RV}.

 An elementary definition of $T^1$  is
\[
T^1(C,0) = 
 \Coker \Theta_n\otimes \mO \to  \Hom_{\mO_n}(I,\mO)
\]
One can compute $\Hom_{\mO_n}(I,\mO)$ from equations and relations. Let $f$ be the row vector
of generators of the ideal of $C$ and $r$ the matrix of relations, so $fr=0$. An infinitesimal
deformation is of the form $f+\varepsilon f'$ and flatness requires that 
$fr=0$ can be lifted to 
\[
(f+\varepsilon f')(r+\varepsilon r') \equiv 0 \pmod {\varepsilon^2}
\] 
This
says that $r^t( f')^t\cong 0  \pmod {f}$. Furthermore 
$ \Theta_n\otimes \mO$ is the free $\mO$-module generated by the derivations $\partial_i=\frac{\partial}{\partial x_i}$.
Computing $f'$, that is syzygies over the quotient ring  of the transpose of the relation matrix,
can be done with a computer algebra system.
In particular, to show that a homogeneous singularity has no deformations of
(strictly) negative weight, it suffices to compute
generators of the normal sheaf $ \Hom_{\mO_n}(I,\mO)$ up to degree $-1$.
If there as many generators
as the number of variables, then there are no deformations of
(strictly) negative weight, besides the trivial ones
coming from the $\partial_i$.

For monomial curves the elementary approach suffices to give
 a dimension formula for the graded parts of $T^{1}(\k[\S])$. 
The ideal of 
$C_{\S}:=\{(t^{a_1},\dots,t^{a_n})\,;\,t\in \k\}\subset\mathbb{A}^{n}$
can be generated by binomials $f_i$ of the form
\[
f_i:=x_{1}^{\alpha_{i1}}\dots x_{n}^{\alpha_{in}}-x_{1}^{\beta_{i1}}\dots x_{n}^{\beta_{in}}
\]
with $\alpha_i\cdot\beta_i=0$. As usual, the weight of $f_i$ is $q_i:=\sum_j a_j\alpha_{ij}=\sum_j a_j\beta_{ij}$.
For each $i$, let $v_i:=(\alpha_{i1}-\beta_{i1},\dots,\alpha_{in}-\beta_{in})$
be the vector in $\k^{n}$ induced by $f_i$.

\begin{thm}[{\cite[Thm. 2.2.1]{Bu80}}]\label{thm:monomialT1}
Let $\k[\S]$ be the semigroup ring of a numerical semigroup $\S=\langle a_1,\dots,a_n\rangle$.
For $\ell\in\mathbb{Z}$ let $A_{\ell}:=\{i\in\{1,\dots,n\}\mid a_i+\ell\notin\S\}$
and let  $V_{\ell}$ be the vector subspace of $\k^{n}$ generated by the vectors $v_i$
such that $d_i+\ell\notin\S$. Then for $\ell\notin\End(\S)$
\[
\dim T^{1}(\k[\S])_{\ell}=\# A_\ell-\dim V_{\ell}-1
\]
while $\dim T^{1}(\k[\S])_s=0$ for  $s\in\End(\S)$ .
 \end{thm}

For reducible quasi-homogeneous singularities we need a more sophisticated
description of $T^1$. We follow \cite{RV} and \cite{Gr0}.

Let as before $K$ be the total ring of fractions of $\mO$.
The exact sequence $0\to \mO \to K \to K/\mO \to 0 $ gives an exact sequence of
cotangent modules
\[
0 \to T^0(\mO,\mO)  \to T^0(\mO,K)\to T^0(\mO,K/\mO)
\to T^1(\mO,\mO)  \to 0
\]
where $T^1(\mO,K)=0$ because $\mO$ is generically smooth over $\k$, and
$T^0(\mO,K)\cong K$.  Here $T^1(\mO,\mO)$ is the previously defined $T^1(\mO)$.
For any $\mO$ module $M$ one has $T^0(\mO,N)=\operatorname{Der}_k(\mO,N) = 
\Hom_{\mO}(\Omega,N)$. 

If $I=\langle f_1,\dots,f_k\rangle$ is  the ideal of $(C,0)$, then 
\[
T^0(\mO,K/\mO) =\Hom_{\mO}(\Omega,K/\mO)\subset 
\Hom_{\mO}(\Omega^1_n,K/\mO)\cong (K/\mO)^n
\]
  is isomorphic
to the kernel of the map ${\partial f}\colon(K/\mO)^n\to (K/\mO)^k$ 
induced by the Jacobi matrix
$\left(\frac{\partial f_j}{\partial x_i}\right)$.

For quasi-homogeneous $(C,0)$ all the modules are graded and the maps have degree 0.
The module $(K/\mO)^n$ has the vector fields $\frac{\partial }{\partial x_i}$ as basis,
so $(K/\mO)^n_{\ell} = \oplus_i  (K_{\ell+w_i} /\mO_{\ell+w_i} )$, while 
$(K/\mO)^k_{\ell} = \oplus_j  (K_{\ell+q_j} /\mO_{\ell+q_j} )$. We embed $\overline \Omega$ in $K$
via the map $(dt_1,\dots,dt_r)\mapsto(t_1,\dots,t_r)$.  Then for quasi-homogeneous curves the
image of $\Omega$ is $\m$ \cite[Lemma 2.2]{Gr}.
Therefore $T^0(\mO,\mO)\cong \Hom_\mO(\m,\mO)=\{a\in K\mid a\m\subset \mO\}$.

This gives us the exact sequence
\begin{equation}\label{eq:exact} 
 0 \to \Hom_\mO(\m,\mO)_{\ell} \to K_{\ell} \to (\ker{\partial f}) _{\ell}\to T^1_{\ell} \to 0
 \end{equation}
 We apply this in particular to homogeneous curves, so $w_i=1$ for all $i$.
 \begin{cor}\label{cor:Tplusnul}
 For homogeneous $(C,0)$ we have
 $T^1_{\ell}=0$ if $K_{\ell+1}=\mO_{\ell+1}$.
 In particular, for $L_r^n$ with  $\binom{n+d-2}{d-1} <r \leq \binom{n+d-1}{d}$ in general
  position we have $T^1_{\ell}=0$ for $\ell\geq d-1$.
 \end{cor}

\section{Buchweitz criterion}
\subsection{}
Buchweitz defined an invariant for curve singularities that varies 
semicontinously in families. Being published only  in a preprint and as second
part of his Thèse \cite{Bu}, this invariant is not very well known. The only place in the literature
we are aware of, that treats Buchweitz' original approach is \cite[\pargr \,6]{KW}.
Buchweitz' definition uses Noether normalisation and differents for finite mappings,
in particular the Kähler and Dedekind differents (see  e.g.  
\cite[\href{https://stacks.math.columbia.edu/tag/0DWH}{Tag 0DWH}]{SP}), and 
actually he defines several invariants.
 Here we use only the Dedekindian one and use 
an alternative form \cite[Remark after Lemma 3.1]{Bu}.

We embed $\omega$ in $K$ via the map $(dt_1,\dots,dt_r)\mapsto(1,\dots,1)$. For any 
two $\mO$-ideals $\mathfrak a,  \mathfrak b$ in $K$  containing a non-zero divisor one has
\[
\Hom_\mO(\mathfrak b,\mathfrak a)=\mathfrak a : \mathfrak b=
\{x\in K\colon x \mathfrak b  \subset \mathfrak a\}, 
\]
where the homomorphism is given 
by multiplication with  $x\in K$. As $\mO\subset \omega$, we get under this
identification $\Hom(\omega,\mO)\subset \mO$.
\begin{defn}
The $k$-th normalised Dedekind invariant of the curve $(C,0)$ is
\[
d_k=d_k(C) = \dim \Coker \Hom(\omega^{\otimes k},\mO) \to \mO
\]
\end{defn}

The following proposition is in \cite{Bu}  a consequence of the definition.
Here we give  a direct proof.

\begin{prop}
For Gorenstein curves $d_k=2k\delta$.
\end{prop}


\begin{proof}If $C$ is Gorenstein, the module $\omega$ is free, with 
generator $\alpha$ satisfying $v_P(\alpha) = -c _P$ for all $P\in n^{-1}(0)$, and
$\Hom(\omega^{\otimes k},\mO)$ is generated by the map which sends
$\alpha^k$ to $1\in \mO$. This corresponds to an element $f^k\in \mO$, where
$f$ generates the conductor ideal $\mathfrak c$ in $\widetilde\mO$ and the
image of $\Hom(\omega^{\otimes k},\mO)$ is generated by $f^k$ 
as $\mO$-module. As $\dim \widetilde\mO/\mathfrak c = 2\delta$ for 
 Gorenstein  curves \cite[IV.11]{Se},
we get $\dim  \mO/f^k\mO = \dim \widetilde\mO/\mathfrak c^k -
\dim \widetilde\mO/\mO+\dim f^k \widetilde\mO/f^k\mO= 2k\delta-\delta+\delta
=2k\delta$.
\end{proof}


The essential result proven by Buchweitz is the semicontinuity of the invariant
$b_k:=d_k-2k\delta$. His proof uses Noether normalisation. It is written 
in the context of complex analytic curve germs, but it is remarked that the proof
also works in the formal category of complete Noetherian rings.

\begin{thm}[Buchweitz]
Let $\pi \colon \mathcal C \to T$ be a 1-parameter deformation of a reduced 
curve singularity $C$, then 
$b_k(C) \leq b_k(\mathcal C_t)$,
where $\mathcal C_t$ is the generic fibre, and the invariant for $\mathcal C_t$ is the
sum of the invariants of its singular points.
\end{thm}

\begin{cor}\label{cor:rb-crit}
If $C$ is deformable into a curve with at most Gorenstein singularities (in particular,
if $C$ is smoothable), then $d_k\leq 2k\delta$ for all $k\in \N$.
\end{cor}

\begin{ex} \label{ex:buchcurv}
The smallest example of a monomial curve, where $d_1>2\delta$, is the curve
with semigroup $\langle 13,\dots,18, 20,22,23 \rangle$. It occurs already
in the original paper by Buchweitz \cite{Bu}.
In this case $c=26$, and $\delta=16$.
A basis for $\omega_X/\Omega^1_{\widetilde X}$ is
\[
\frac{dt}{t^{26}}, \frac{dt}{t^{25}},\frac{dt}{t^{22}},\frac{dt}{t^{20}},
\frac{dt}{t^{13}},\dots,\frac{dt}{t^{2}}
\]
where the first four are generators of $\omega$,
so if $\varphi \in \Hom(\omega,\mO)$ maps $\frac{dt}{t^{26}}$ to $t^l$, then 
$\varphi(\frac{dt}{t^{25}})=t^{l+1}$, $\varphi(\frac{dt}{t^{22}})=t^{l+4}$ and
$\varphi(\frac{dt}{t^{20}})=t^{l+6}$. This means that with $t^l$
also $t^{l+1}$, $t^{l+4}$ and $t^{l+6}$ have to be elements of $\mO$.
The smallest possible value for $l$ is therefore $14$, followed by $16$, $22$ and 
all values from $26$ on. Such a $\varphi$ corresponds to $t^{l+26}\in \mO$
and therefore $d_1=52-3-16=33>2\cdot16=2\delta$.

We describe some deformations of the curve.
Its equations can be given in a redundant way by $2\times 2$ minors of the following
matrix, where the dots cannot be filled in for reasons of degree. Only minors
that do not contain dots lead to equations.
\[
\begin{vmatrix}
x_{13}& x_{14}& x_{15} & x_{16} & x_{17} & x_{18} & x_{20} & x_{22} & x_{23}\\
x_{14}& x_{15} & x_{16} & x_{17} & x_{18} & .        & .      & x_{23}& .\\
x_{15} & x_{16} & x_{17} & x_{18} & .        & x_{20} & x_{22}& .   &. \\
 x_{16} & x_{17} & x_{18} & .        & x_{20} &. & x_{23}& .   &x_{13}^2\\
 x_{17} & x_{18} & .        & x_{20} & . &x_{22}& .   &x_{13}^2 & x_{13}x_{14}
 \end{vmatrix}
\]
Perturbations of the entries of the matrix lead to deformations of the curve.
In this way we can describe the deformation of lowest degree, by
replacing   $x_{13}^2$  with $x_{13}^2+sx_{13}$ and replacing 
 $x_{13}x_{14}$  with $x_{13}x_{14}+sx_{14}$.  
The  tangent cone of the deformed curve for $s\neq 0$ is obtained by 
 replacing  $x_{13}^2$ with $sx_{13}$,
 $x_{13}x_{14}$ with $sx_{14}$. The resulting  equations describe a  $L_{13}^9$.
 Note that this deformation lies on the cone over a  projected  rational normal curve,
 with $x_{13}=u^{10}$, $x_{14}=u^9v$,  $x_{15}=u^8v^2$,  $x_{16}=u^7v^3$, 
  $x_{17}=u^6v^4$, $x_{18}=u^5v^5$,  $x_{20}=u^3v^7$,   $x_{22}=uv^9$ 
  and $x_{23}=v^{10}$.
  The deformation is then the image of the curve $u^{23}+su^{13}-v^{13}$.
As the tangent cone has no deformations of positive weight, the deformed singularity
is in fact isomorphic to this tangent cone. By semicontinuity of Buchweitz' invariant this is 
a particular $L_{13}^9$, for which $d_1>2\delta$. A direct computation for $s=1$ shows
that $T^1$ of this $L_{13}^9$
is concentrated in degree 0, of dimension $24$ equal to the number of moduli of points. By 
Theorem \ref{thm:smooth-gale} we know that the general 
$L_{13}^9$ is not  smoothable. 
We have here however a rather simple to describe specific example. 
\end{ex}

The computation of $d_1$ can be done in general for $L_n^r$ with $r\leq \binom{n+1}2$.

\begin{prop}
Let $L_r^{r-g}$ be the cone over $r$ points in $\P^{r-g-1}$ in uniform position, with
$r\geq g + \frac{1+\sqrt{8g+1}}2$, and $g\geq 4$.
If the $L_r^{r-g}$ is smoothable, then 
the Gale transform of the $r$ points  in $\P^{g-1}$ imposes at most $3g-3$ conditions on quadrics.
\end{prop}

\begin{proof}
The assumption on $r$  is equivalent to $r\leq \binom{r-g+1}2$, so by
Lemma \ref{lem:delta} we have $\delta (L_r^{r-g})=r+g-1$. 
Let $\mathcal I$ be the image of $\Hom(\omega,\mO)$ in $\mO$. 
As $\mO_2=K_2$ we have $\m^4\subset \mathcal I\subset \m^2$. 
Let $s=\dim \m^3/\mathcal I$. Smoothability implies  
$d_1\leq 2\delta$, so $1+r-g+r+s\leq 2r+2g-1$, that is $s\leq 3g-3$.

Let $P$ be the matrix describing the points and $Q$ the matrix describing the Gale transform,
so $PQ^t=0$. Then $Q$ is also the matrix of coefficients of the generators of $\omega$,
considered as elements of $\overline\Omega\otimes K$. Multiplying  
a generator $\alpha$ of $\omega$ with $f=(f_1,\dots,f_r)\in \mO_3$  leads to a vector of coefficients,
which should be a linear combination of the rows of $P$, so its transpose should lie
in the kernel of the matrix $Q$. We obtain therefore the matrix equation $QF Q^t=0$, where 
$F=\operatorname{diag}(f_1,\dots,f_r)$. Considered
as linear equations for the $f_1,\dots,f_r$ the coefficient matrix becomes, after deleting identical rows,
 a $\binom{g+1}2
\times r$ matrix, where the columns are the coordinates of the $r$ points in the Gale transform
under the second Veronese embedding $v_2$. This is the matrix describing the conditions
on quadrics, and $s$ is its rank.
\end{proof}

By this proposition the Gale transform of the $r$ points lies on at least $(g-2)(g-3)/2$
linearly independent quadrics.
By Max Noether's Theorem  a non-hyperelliptic canonical curve of genus $g$
lies on exactly this number of quadrics \cite[p. 117]{ACGH}. So the previous
result is weaker than Theorem  \ref{thm:smooth-gale}, but it suffices to show
that the $L_{10}^6$ of Example \ref{ex:L106} is not smoothable. Therefore
this non-smoothability also follows from Buchweitz' criterion.
 
 \begin{ex}
 The following example  of a non-smoothable monomial curve with $d_1=2\delta$ but $d_2=4\delta +1$ is due to Komeda \cite[Example 2.4 (2)]{Ko}. Consider
 the semigroup
 \[
 \langle r,r+1,\dots,2r-8,2r-7, 2r-4,2r-3\rangle
 \]
 of genus $g=r+3$ with conductor $2r$. The pole orders of differentials are
 \[
 2r, 2r-1,\;2r-4,2r-5,\;r,r-1,\dots,2
 \]
One checks easily that for $r\geq 12$ multiplication with 
\[t^{3r},\dots,t^{4r-12},t^{4r-8},t^{4r-4},t^{4r},t^{4r+1},
\dots
\]
 maps $\omega$ into $\mO$, giving that $d_1=3r+9 - (r+3)=2\delta$. The pole orders of
 quadratic differentials, relevant for the computation of $d_2$ are $4r,\dots, 4r-10$
 with $4r-3$ and $4r-7$ excepted. Therefore only multiplication by $t^{5r},\dots, t^{5r-17}$ and 
 from $t^{6r}$ onwards works, giving $d_2=4r+13$ for $r\geq 16$. 
 
  The  monomial curve lies on the cone $x_r=u^{r-3}$, 
  $x_{r-1}=u^{r-4}v$, \dots, $ x_{2r-7}=u^4v^{r-7}$, $x_{2r-4}=uv^{r-4}$ 
  and $x_{2r-3}=v^{r-3}$ as
  the image of $u^{2r-3}=v^r$.
  It deforms into the  $L_r^{r-4}$ one the cone, which is the image of $u^r=v^r$. The point matrix 
  $P$ involves powers of the $r$-th roots of unity and the matrix $Q$ of the Gale transform also.
  The Gale transform of the $r$ points lies on a non-degenerate quadric, but not on a curve
  of type $(3,3)$ if $r\geq 16$. Indeed, as this $L_r^{r-4}$ with $r\geq 16$ is not smoothable,
 the Gale transform cannot lie on a canonical curve; note that the lower bound 16 is sharp, as
 15 points determine a curve  of type $(3,3)$.
 \end{ex}
 
\subsection{} The known non-smoothable monomial curves have non-trivial deformations of
negative weight. By openness of versality the curves occurring in this way are also not
smoothable. Buchweitz' original example is the smallest, with $g=16$, but there is
an example with smaller embedding dimension having $g=17$ \cite{Ko}.  For computations with
explicit equations this is easier.

\begin{ex}
Consider the 
semigroup of embedding dimension 8 with generators
$\langle 13,\dots,18, 21,23 \rangle$ \cite[Example 2.1 (2)]{Ko}.
We do not give equations for the curve here, as their structure is  similar to that in Example 
\ref{ex:buchcurv}. The equations can be given in \textit{rolling factors format}, see
e.g., \cite[Ch. 12]{St4}.
There are quadratic equations expressing that the curve lies 
 on the cone over a  projected  rational normal curve,
 with $x_{13}=u^{10}$, $x_{14}=u^9v$,  $x_{15}=u^8v^2$,  $x_{16}=u^7v^3$, 
  $x_{17}=u^6v^4$, $x_{18}=u^5v^5$,  $x_{21}=u^2v^8$ 
  and $x_{23}=v^{10}$. Furthermore there are rolling factors equations obtained 
  from the equation $u^{23}-v^{13}=0$ of the curve on the cone. 
  
  A  rolling factor deformation is induced by $u^{23}+su^{13}-v^{13}$,
   to a $L_{13}^8$ lying on the cone.
  More rolling factor deformations can be obtained 
   from the deformation $u^{23}+su^{13-k}v^k-v^{13}$.
    Writing this last expression as
 \[
 \frac{(u^{10+k }+sv^k)(su^{13-k}-v^{13-k})}s-\frac{u^{10+k}v^{13-k}}s
 \]
  we see that the curve consists of $13-k$ smooth branches and a singular branch of  multiplicity $k$.
  
  In particular, for $k=2$ a direct computation with Singular  \cite{Sing} shows that it is a singularity
  with $\dim T^1 = 28$, that is the same dimension as for the general $L_{13}^8$. 
  This is maybe unexpected, but one has to keep in mind 
 that the singularity
  is not quasi-homogeneous, and there is no $\G_m$-action on the base space.
  
  If we take $s=-1$ and consider the $11$ lines $u^{11}+v^{11}$ together with the curve,
  which is the image of $u^{12}-v^2$, so the curve parametrised by
  $(x_{13},\dots,x_{23})=(t^2, t^3,t^4,t^5,t^6,t^7,t^{10},t^{12})$, 
  then it is a trivial deformation of the previous curve, but
  the computation of $T^1$
  does not finish in reasonable time. However it is possible to compute the  generators of the 
  normal sheaf  $ \Hom_{\mO_n}(I,\mO)$,
   and conclude that there are no non-trivial infinitesimal deformations 
  with linear part, and that the number of generators apart from those coming from coordinate 
  transformations is equal to $28$. 
  The same holds for a curve consisting of an ordinary cusp and $11$ lines in general
  position. This shows:
  \begin{prop}\label{prop:cusplines}
  The general curve consisting of  an ordinary cusp and $11$ lines in general
  position through the origin in    $\A^8$ is not smoothable. It deforms only into 
  curves $L_{13}^8$.
  \end{prop}
\end{ex}

\subsection{}
For monomial curves the Buchweitz criterion can be expressed in terms of the semigroup.
Recall that $\L$ is the set of gaps. Denote by $k\L$ the $k$-fold sumset $\L+\dots+\L$
and by  $|k\L|$ its cardinality.

\begin{lemma}[Buchweitz] For  a monomial curve of multiplicity at least three
\[ \dim \Coker \left\{\Hom(\omega^{\otimes k},\mO) \to \mO\right\} =
|(k+1)\L|+(2k+1)-\delta\]
\end{lemma}
\begin{proof}
Multiplication by  $t^l\in \mO$ with $l\geq kc$
is the homomorphism  $\vp_l\colon \omega^{\otimes k} \to
\widetilde {\mO}$ given by $\vp_l\left(\frac{(dt)^k}{t^a}\right)=t^{l-a}$. The image of this map
is not contained in $\mO$ if and only if
$l-a\notin \Gamma$ for some $a$. By the description of $\omega$
we can write $a-k=(a_1-1)+\dots+(a_k-1)$
with $a_i-1\notin \Gamma$. If for some $i$ one has $a_i-1\notin \L$, that is $a_i\leq0$,
then $a\leq (k-1)c$ and $\vp_l\left(\frac{(dt)^k}{t^a}\right)=t^{l-a}\in \mO$ as $l-a\geq c$.
Therefore the image of $\vp_l$ is not  contained in $\mO$ if and only if
$l-k$ can be written as the sum of  $k+1$ elements  of $\L$.
For $l=2k+1,\dots, kc-1$ we have that $l-k=k+1,\dots,kc-(k+1)$ and all these
numbers belong to $(k+1)\L$: as $c-1$ is the maximal element of $\L$, this statement
follows by induction from the case $k=1$, which is true by \cite[4.2]{Bu}, \cite[Thm. (1.3)]{Ol}. 

So $d_k+\delta$ is equal to the number of $t^l\in\widetilde {\mO}$ with $l<kc$
plus the number of elements in $(kc-k+\N)\cap(k+1)\L$; this sum is equal
to $|(k+1)\L|+(2k+1)$.
\end{proof}

\begin{cor}
For  a monomial curve of multiplicity at least 3 one has
$d_k>2k\delta$ if and only if $|(k+1)\L|>(2k+1)(\delta-1)$.
\end{cor}

This is the form in which the Buchweitz criterion is mostly cited, but  
the original form $d_k>2k\delta$ is wider applicable.

\subsection{Weierstrass semigroups}
For  a smooth projective pointed curve $(C,P)$ of genus $g>1$ defined
over $\k$ the set of nonnegative integers $n$, such that there is a rational
function on $C$ whose pole divisor is $nP$, form a semigroup $\S$, the Weierstrass semigroup
at $P$. By Riemann-Roch
the set $\L$ of positive integers that are not in $\S$ (the set of gaps) has size exactly $g$.
The point $P\in C$ is a  Weierstrass point if its 
semigroup is different from the ordinary one $\{0, g+1, g+2, \dots\}$.

Let $\mathcal M_{g,1}^{\S}$ be the moduli space of smooth  pointed curves of genus $g$
whose Weierstrass semigroup at the marked point is $\S$. 
Pinkham \cite[Section 13]{Pi} observed that it is 
related to the negative part $\mathcal C^-\to B^- $
of the versal deformation of  the monomial curve $C_\S$. The $\mathbb G_m$-action on
the curve induces a  $\mathbb G_m$-action  on the base space $B$ and on $B^-$.

\begin{thm}[Pinkham]
Let $B^-$ be the base space in negative degrees of the monomial curve $C_\S$ and 
denote by $B^-_\text{sm}$ the open subset of $B^-$ given by the points
with smooth fibres. Then  $\mathcal M_{g,1}^{\S}$ is isomorphic to 
$B^-_\text{sm}/\mathbb G_m$.
\end{thm}

This means that $C_\S$ is negatively smoothable if and only if $\S$ is a Weierstrass semigroup.
Buchweitz'  condition  $|(k+1)\L|>(2k+1)(\delta-1)$  implies that $C_\S$ is not smoothable 
and therefore not negatively smoothable, so $\S$ is not a Weierstrass semigroup.
This can also be seen directly by Riemann-Roch. Suppose that $P$ is a Weierstrass point
on a smooth curve $C$. If $n$ is a gap, then there exists a regular differential
$\alpha$ with a zero of order $n-1$.  If $\alpha_1,\dots,\alpha_k$ are differentials
with zeros of order $n_i-1$, $i=1,\dots,k$, then their product defines a
$k$-fold differential with a zero of order $\sum n_i-k$. By Riemann-Roch 
$h^0(C,\Omega^k)=(2k+1)(g-1)$, so the number of zero orders can at most  be $(2k+1)(g-1)$.
Therefore $|(k+1)\L|\leq(2k+1)(g-1)$ has to hold.
In the literature this argument is most often given to prove Buchweitz' criterion,
with the notable exception of \cite{KW}.

\subsection{}
The known non-smoothable monomial curves have multiplicity at least 13.
Komeda has found examples of non-Weierstrass semigroups with multiplicity 8 and 12
\cite{Ko2}. These monomial curves are therefore not  negatively smoothable, but  they
might be smoothable. We show that this is indeed so in the simplest case.
This gives an  example of a smoothable irreducible
quasi-homogeneous curve singularity which is not   negatively smoothable. Earlier Pinkham gave an example of a reducible curve \cite[p. 70]{Pi2}.

\begin{ex}\label{ex:negsmooth}
The semigroup $\S=\langle8,12,18,22,51,55\rangle$ is not a Weierstrass semigroup
\cite[Example 5.1]{Ko2}.
The ideal of  the monomial curve
$C_\S$  is generated by 13 polynomials, and the deformation in negative weight unobstructed, with dimension $57$. Fortunately it
suffices to write down  the deformation for generators of $T^1$ as $\mO$-module, the other
deformations can be obtained by substitution.
We write the deformation in
rolling factors format, with 9 equations given (non-minimally) by the following determinantal
and two pairs of remaining equations:
\[
\begin{bmatrix}
x_{8}   &   x_{12}  & x_{18}  & x_{22}  & x_{51} & x_{55} \\
x_{12}  &  x_8(x_8+s_8) & x_{22} &   x_{18}(x_8+s_8)  & x_{55}& x_{51}(x_8+s_8) 
\end{bmatrix}
\]
\begin{gather*}
x_{12}^3-x_{18}^2+  x_{22} s_{14}+x_{18} s_{18}+x_{12} s_{24}+x_8 s_{28},  \\                                                                  
\begin{split}
x_{12}^2 x_8 (x_8+s_8)-x_{22} x_{18}+x_{18} (x_8+s_8) s_{14} \qquad\\+  x_{22} s_{18}
+x_8 (x_8+s_8) s_{24} +x_{12} s_{28},\end{split} \\
x_{22}^3 x_{18}^2-x_{51}^2 +   x_{55} s_{47}+x_{51} s_{51}+x_{22} s_{80}+x_{18} s_{84}+x_{12} s_{90}+x_8 s_{94}, \\
\begin{split}
 x_{22}^4 x_{18}-x_{55} x_{51}+ x_{51} (x_8+s_8) s_{47}+x_{55} s_{51}+x_{18} (x_8+s_{8}) s_{80}\qquad\\+x_{22} s_{84}+x_8 (x_8+s_{8}) s_{90}+x_{12} s_{94}.\end{split}
\end{gather*}
From these equations one sees immediately that the curve is not negatively smoothable: at the origin the Jacobian matrix has rank at most  4, as only in four equations the $x$-variables occur linearly.
But the singularity is smoothable, because the general fibre of this deformation
has only a hypersurface singularity at the origin. To see this, 
take the 1-parameter deformation  $s_8=s^4$, $s_{18}=s^9$, $s_{94}=s^{47}$,  while the
other deformation variables are zero.
A computation shows that for $s\neq0$  there is indeed only one singularity, at the origin; the
Jacobian matrix has rank 4.
The first two additional equations become $x_{18}(s^9-x_{18})=-x_{12}^3$,                                                                   
$x_{22}(s^9-x_{18})=x_{12}^2 x_8 (x_8+s^4)$ with $(s^9-x_{18})$ a unit in the local ring.
This allows to eliminate $x_{18}$ and $x_{22}$. The last two equations  then allow
elimination of $x_8$ and $x_{12}$. What remains is $x_{55}^2=ux_{51}^2$ with  $u$ a unit,
so the curve has an ordinary double point.
\end{ex}

\section{Large families}
\subsection{}
The second general method for showing non-smoothability is based on exhibiting large families of 
singularities, too large to be in the closure of the locus of smooth ones. Iarrobino \cite{Ia}
used it to study  zero-dimensional schemes.
The first examples of non-smoothable curve singularities were given by Mumford 
\cite{Mu},  using this method. He constructed a family of singular complete curves 
which is too large to lie in the closure of the moduli space of smooth curves of the genus in question.
Greuel \cite{Gr} used his version of Deligne's formula for the dimension of smoothing components
for quasi-homogeneous singularities to analyse Mumford's examples, and those of Pinkham \cite{Pi}.
In fact, he gave the following general criterion.

\begin{prop}\label{prop:greuel}
Let $\pi\colon \mathcal C \to T$ be a deformation with 
singular section $\sigma\colon T\to \mathcal C$ of the curve singularity
$C=\mathcal C_0$. If $\mathcal C_t$ is not isomorphic to $C$ for $t\neq0$
and $T$ is irreducible of dimension $\dim T \geq e(C)$, then there is a dense
open subset $T'\subset T$ such that $(\mathcal C_t,\sigma(t))$ is not 
smoothable for $t\in T'$.
\end{prop}
Indeed, the image of $T$ in the versal deformation has dimension $\dim T \geq e(C)$
and the image cannot be a smoothing component, as there are no smooth fibres
over $T$.

\subsection{Lines through the origin in general position}
 We first analyse in detail the examples of Pinkham and Greuel \cite{Pi,Gr}.  
 Consider the singularity $L_r^n$ of $r$ lines in $\A^n$ through the origin in general position
 (see Section \ref{sssc:lines}); it is the 
 cone over $r$ points in $\P^{n-1}$. 
The number of moduli of $r$ points in $\P^{n-1}$
is  $(r-n-1)(n-1)$.  
%

 For fixed $n$ the curves $L_n^n$, $L_{n+1}^n$ and $L_{n+2}^n$ are always
 smoothable. 
 We now consider $r\geq n+3$ and 
 determine the Deligne number for generic curves.
 
Let $d=d(n,r)$.
By Lemma \ref{lem:delta}
\[
\delta(L_r^n) = dr-\binom{n+d-1}{d-1}
\]
For generic curves the type $t$ is by Proposition \ref{prop:type} given by
\[\textstyle
t= \max \{ r-\binom{n+d-2}{d-1}, \binom{n+d-3}{d-1}-(n-2)\left( r-\binom{n+d-2}{d-1}\right)\}\]
 Therefore the Deligne number $e=\mu+t-1=2\delta-r+t$
 is
 \[\textstyle
e= \max \begin{cases}
 2dr-2\binom{n+d-1}{d-1}-\binom{n+d-2}{d-1}\\ 
 (2d+1-n)r-2\binom{n+d-1}{d-1}+\binom{n+d-3}{d-1}+(n-2)\binom{n+d-2}{d-1}
\end{cases}\]

 We first  consider the condition $(r-n-1)(n-1) \geq e$ for $r$ 
such  that $ t = r-\binom{n+d-2}{d-1}$, which holds for most $r$. Then also $e$ is given by the
first alternative. The condition that the number of moduli is at least $e$
translates into
$r(n-1)-(n^2-1)\geq  2dr-2\binom{n+d-1}{d-1}-\binom{n+d-2}{d-1}$, which we rewrite as
\begin{equation} \label{formula-rnd}
(n-2d-1)r\geq (n^2-1)-2\binom{n+d-1}{d-1}-\binom{n+d-2}{d-1}
\end{equation}
We determine a lower bound for $r$, in the interval $n<r\leq \binom {n+1}2$, so $d=2$.
We find $(n-5)r\geq n^2-3n-3=(n+2)(n-5)+7$. 

For an upper bound we distinguish between even and odd $n$. 
For odd  $n=2m+1$  the condition \eqref{formula-rnd} is obviously satisfied if $d=\frac{n-1}2=m$ and $m\geq 3$.
We claim that it is no longer satisfied for  $r=\binom{n+m}{m+1}=\binom{3m+1}{m+1}$.
We determine where on the interval $\binom{n+m-1}{m} <r \leq \binom{n+m}{m+1}$
formula  \eqref{formula-rnd} ceases to hold:
\[
-2r\geq 4m(m+1) -2\binom{3m+1}{m}-\binom{3m}{m}
\]
gives 
\begin{equation}\label{eq:odd}
r\leq \binom{3m+1}{m}+ \frac12\binom{3m}{m}-2m(m+1)\eqdef M(2m+1)
\end{equation}
where the right hand side is indeed less than $\binom{3m+1}{m+1}$.
If we use instead the second formula for $e$, which holds for $r-\binom{3m}{m}$ small,
then we obtain a condition of the form
$(2n-2d-2)r=(2m-2)r \geq F(m)$ with $F(m)$ an explicit expression independent of $r$, which
gives the correct condition for $r=\binom{3m}{m}$, and we have already noticed that it is satisfied
in that case, if $m\geq3$. The maximal   $r$  lies therefore in the range where we need the
first expression.

For even $n$ this method gives no non-smoothability result for $n=4$. 
Let $n=2m$.  For $d=m$ we have that the left hand side of formula \eqref{formula-rnd}
is negative, but for $n\geq 8$ the formula holds on the whole interval.
For $n\geq 8$  we take $d=m+1$
and find in a similar way as above
\begin{equation}\label{eq:even}
r\leq \frac23\binom{3m}{m}+ \frac13\binom{3m-1}{m}-\frac{4m^2-1}3 \eqdef M(2m)
\end{equation}
For $n=6$ we need $d=m$ and we get $r\leq 2 \cdot\binom{8}{2}+\binom{7}{2}-35 = 42$. 

From Proposition \ref{prop:greuel} we get
\begin{prop}\label{prop:interval}
Let $M(n)$ for $n\geq 7$ be given by the expression
in  \eqref{eq:odd} for odd $n$ and in  \eqref{eq:even} for even $n$ and set
$M(6)=42$. For $n\geq 6$ the generic $L_n^r$ in uniform position is not smoothable
if $n+2+\frac6{n-5}<r\leq M(n)$.  
\end{prop}

For $d=2$ Theorem \ref{thm:smooth-gale} gives an exact criterion for smoothability:
the generic $L_r^n$ is smoothable if through the  Gale transform of the $r$ corresponding 
points passes a canonical curve of genus $g=r-n$.  We formulate  the condition  $r> n+2+\frac6{n-5}$  
in terms of $g$ as  $r>g+5+\frac 6{g-2}$. According to  \cite{St} $r$ points determine
a canonical curve  when    $r\leq g+5+\frac 6{g-2}$,
except for $g=4$ and $g=6$, when $r\leq g+5$.  For low values of $n$ we have, with this correction,
non-smoothability for the following values of $r$:
\[
\begin{array}{c||c|c|c|c|c}
n&6         &   7          &  8          &  9       &  10\\  \hline
\vphantom{A^{A^A}} r&\{10,12\}\cup[15,42]&\{11\}\cup[13,138] & [12,419] &  [13,922]  &  [14,2636] 
\end{array}
\]
For $r$ larger than $M(n)$ the curve  $L_r^n$ has 
deformations of positive weight, so the homogeneous curves are not the most general
on the equisingularity stratum. 
We compute the dimension of the tangent space.

\begin{prop}\label{prop:tplus}
For $\ell >0$
\[\dim T^1_{\ell} \geq 
\max \left\{0, (n-1)\left(r-\binom{n+\ell}{\ell+1}\right)-\binom{n+\ell-1}{\ell+1}\right\}\]
\end{prop}
\begin{proof} Let $\binom{n+d-2}{d-1} <r \leq \binom{n+d-1}{d}$. Then there are
$k_1= \binom{n+d-1}{d}-r $ equations of degree $d$ and possibly $k_2$ equations
of degree $d+1$. Furthermore $K_\nu=\mO_\nu$ for $\nu\geq d$. Therefore $(K/\mO)^k_{\ell}=
(K_{d+\ell}/\mO_{d+\ell})^{k_1}\oplus (K_{d+1+\ell}/\mO_{d+1+\ell})^{k_2}=0$
for $\ell >0$. The exact sequence \eqref{eq:exact} reduces to 
\[
 0 \to \Hom_\mO(\m,\mO)_{\ell} \to K_{\ell} \to \big(K_{\ell+1}/\mO_{\ell+1}\big)^n \to T^1_{\ell} \to 0
 \]
 Here $\Hom_\mO(\m,\mO)_{\ell} =\{ a\in K_{\ell}\mid a \m_1\subset \mO_{\ell+1}\}$. The map
 $K_{\ell}\to \big(K_{\ell+1}/\mO_{\ell+1}\big)^n$ is induced by the Euler vector field, so
it is given by $a\mapsto ([ax_1],\dots,[ax_n])$, where $[ax_i]$ denotes the class of $ax_i$
in $K_{\ell+1}/\mO_{\ell+1}$ and the $x_i$ are the generators of $\m$. This induces
a map $K_{\ell}/\mO_{\ell} \to (K_{\ell+1}/\mO_{\ell+1}\big)^n$ whose kernel consists
of the elements of $\Hom_\mO(\m,\mO)_{\ell}$ not lying in $\m_{\ell}=\mO_{\ell}$.
Therefore $ T^1_{\ell} = \Coker K_{\ell}/\mO_{\ell} \to  (K_{\ell+1}/\mO_{\ell+1}\big)^n$.
The dimension of the image is at most $r-\binom{n+\ell-1}{\ell}$, so
$\dim T^1_{\ell} \geq  n(r-\binom{n+\ell}{\ell+1})-r+\binom{n+\ell-1}{\ell}
=  (n-1)\left(r-\binom{n+\ell}{\ell+1}\right)-\binom{n+\ell-1}{\ell+1}$.
\end{proof}

\begin{prop}\label{prop:notsm}
The general curve on the equisingularity stratum of $L_r^n$  is not smoothable
if $r> n+2+\frac6{n-5}$ for $n\geq6$, if $r>18$ for $n=5$ and $r>30$ for $n=4$.
\end{prop}

\begin{proof}
The deformations of positive weight are obtained 
by deformation of the parametrisation and are not obstructed. Therefore the
equisingularity stratum has dimension $\dim T^1_{\geq0}$.  We compare the growth of the dimension
with the growth of the Deligne number $e$. For $\binom{n+d-2}{d-1} <r \leq \binom{n+d-1}{d}$
the growth of $\dim T^1_{\geq0}$, if $r$ increases by $1$, is $(n-1)(d-2)$ or $(n-1)(d-1)$ with
the second alternative holding  if $(n-1)\left(r-\binom{n+d-2}{d-1}\right)\geq \binom{n+\ell-1}{\ell+1}$.
The same condition determines the growth of the Deligne number, which is
$2d+1-n$ or $2d$.  So if $(n-3)d\geq n-1$, then $\dim T^1_{\geq0}$ stays with growing $r$
larger or equal than $e$ once it is at least equal to $e$.

For $n\geq 6$ the bound follows from Proposition \ref{prop:interval}; it can be improved by 1 for $n=8$.
For $n=4,5$ we need also $\dim T^1_+$. We give the computation for $n=4$. 
With Proposition \ref{prop:tplus} $\dim T^1_0= 3r-15$, $\dim T^1_1\geq  3r-36$ and $\dim T^1_2\geq 3r-70$.
With $d=4$ we have for  $24\leq r \leq 35$  that $e=8r-90$.
We have equality $8r-90=9r-121$ if $r=31$.
\end{proof}

From this result nothing can be said about smoothability of the homogeneous curve.
But the explicit computations for low values of $n$ show that for some values of $r$ there
are no infinitesimal deformations of negative weight, which proves that for these values the general $L_r^n$ is not smoothable.  For $n=4$ this is the case in the 
intervals $[96,105]$ and $[132,150]$. We compute $T^1_{-1}$ in general.

\begin{prop}\label{prop:tmineen}
For $\binom{n+d-2}{d-1} <r \leq \binom{n+d-1}{d}$ one has 
\[\textstyle
\dim T^1_{-1}\geq \max\left\{0, (n-1)r-n-\left(r-\binom{n+d-2}{d-1}\right)\left(\binom{n+d-1}{d}-r\right)\right\}\] For $r=\binom{n+d-1}{d}$  equality holds, that is $\dim T^1_{-1}=(n-1)\binom{n+d-1}{d}-n$.
 \end{prop}

\begin{proof}
We use  the same notation as in the proof of Proposition \ref{prop:tplus}. 
We have $\Hom_\mO(\m,\mO)_{-1}=0$
so the exact sequence \eqref{eq:exact} becomes
\[
0\to K_{-1} \to 
\ker \left  \{
{\partial f}\colon (K_{0}/\mO_{0})^n\to (K_{d-1}/\mO_{d-1})^{k_1}\right\} \to
T^1_{-1}\to 0
\]
The rank of the map $
{\partial f}$ is at most
\[
\min \{\dim  (K_{0}/\mO_{0})^n- \dim K_{-1}, \dim  (K_{d-1}/\mO_{d-1})^{k_1}\}.
\]
Writing $r=\binom{n+d-2}{d-1}+s$ we have  $\dim  K_{d-1}/\mO_{d-1}=s$ and 
 $k_1= \binom{n+d-1}{d}-r= \binom{n+d-2}{d}-s$.
Therefore  $\dim T^1_{-1}\geq \max\{0, (n-1)r-n-s(\binom{n+d-2}{d}-s)\}$.

For $s=0$ the map $
{\partial f}$ is zero, so $\dim T^1_{-1}=(n-1)r-n$.
\end{proof}

\begin{conj}\label{conj:equality} 
For generic $L_r^n$ equality holds in Proposition \ref{prop:tplus} for  $\dim T^1_\ell$, $\ell>0$ and in 
Proposition \ref{prop:tmineen} for $\dim T^1_{-1}$, 
with some exceptions for low values of $r$, where the
existence of smoothings according to Theorem  \ref{thm:smooth-gale} forces $T^1_{-1}$ to be non-zero.
The generic $L_r^n$ is not smoothable in the range of Proposition \ref{prop:notsm}.
\end{conj}

The conjecture is supported by direct computations.
For $n=4$ the dimension of $T^1$ is as predicted for $r$ up to 70, and that of $T^1_{-1}$
for $r\leq 151$. For $n=5$ we have computed the dimension of $T^1$ for $r\leq 50$
and that of $T^1_{-1}$ for $r\leq 62$. As $\dim T^1_{-1}=0$ for $41\leq r \leq 60$,
the generic $L_r^5$ is not smoothable for those $r$.

\begin{ex} The case $n=6$.
The general $L_{10}^6$ and $L_{12}^6$ are not smoothable, and there are no deformations
of negative weight. In both cases $\dim T^1 < e$.
The general $L_{11}^6$ is smoothable, with $\dim T^1=e=24$. The base space is smooth,
but $\dim T^2=5$.
The general $L_{13}^6$ is also smoothable, but with  $\dim T^1=33=e+1$, so $\dim T^1_{-1}=3$
which is  the minimal value according to Proposition \ref{prop:tmineen}. Furthermore
$\dim T^2 = \dim T^2_{-2} =1$ and a computation shows that the base space in negative degrees is 
singular, given by one quadratic equation.
The case $L_{14}^6$ is described in \cite{St3}. 
Here $\dim T^1_{-1}=8$ and the dimension of a smoothing component is one more than the
number of moduli. The general curve
has 16 smoothing components of dimension 36.

The ideal of $L_r^6$ is generated by quadrics for $6\leq r\leq 14$, but from $r=15$ on one needs
also cubics and for $r=21$ the ideal is generated by 35  cubics. 
The general $L_r^6$ with $15\leq r\leq 21$ is not smoothable, but $ \dim T^1_{-1}$ is not zero,
increasing to 99 for $r=21$, as predicted by Conjecture \ref{conj:equality}.
For $r=15$ a computation shows that all deformations
of negative weight are obstructed.
Therefore the base space of the versal deformation is non-reduced.
We expect the same to be true for other $r\leq 21$.

On  the interval $21< r \leq 56$  the ideal is, up to $r=42$, generated by cubics,
from $r=43$ also quartics are needed, and for $r=56$ the ideal is generated by 70
quartics.
From $r=25$ there are deformations of positive degree, with $\dim T^1_1 = 5(r-24)$, 
computed up to $r=36$.
The dimension of $T^1_{-1}$ decreases, in accordance with the conjecture: 
it is $70$ for $r=22$, $43$ for $r=23$, $18$ for $r=24$
and zero for $25\leq r \leq 47$, to increase afterwards. Only for $r=48$ we have been able to compute the dimension, which indeed turns out  to be $18$. 
In the next interval $56<r\leq 126$ the conjecture predicts that there are no deformations
of negative weight for $61\leq r \leq 116$. We computed only the case  $r=84$, where the size
of the  syzygy matrix 
is minimal:  the general  $L_{84}^6$ has no
deformations of negative weight, and is therefore not smoothable.
\end{ex}

In \cite{Gr2} Problem 2.57 reads
\begin{quote}
Do there exist for fixed $n \geq 4$ non-smoothable curves $L_r^n$
if $r$ goes to infinity? It seems unlikely that this is not the case.
\end{quote}
The evidence above makes it very unlikely that this is not  the case. 
The problem was already formulated in \cite{Gr} without the second sentence, but 
with the question:
\begin{quote}
Do there exist smoothable ones?
\end{quote}
More precisely, we ask whether  there exist for fixed $n$ smoothable $L_r^n$
 with  Hilbert function
$H(\ell) = \min\left(r,\binom{n+\ell-1}{n-1}\right)$ and large $d(n,r)$.
No examples are known to us, in fact not even for $d=3$.

\subsection{}
The curve $L_r^n$ with  $r=\binom{n+1}2 +1$ is the first where
$\delta(L_{r}^n)=\delta(L_{r-1}^n)+3$, because $L_{r-1}^n$
is the first where the ideal is generated by
cubics only and therefore  the intersection multiplicity with a new  line in general position 
is equal to three. It is also possible to construct
a curve  $S_{r}^{n,n+1}$ with $\delta(S_{r}^{n,n+1})=\delta(L_{r-1}^n)+2$ 
having  smooth branches,
where each branch has as tangent line a line of the $L_{r}^n$.
To achieve this we  replace
the last line by a parabola in $\A^{n+1}$, tangent to $\A^n$. 
If the line in $L_{r}^n$ is parametrised by
$(x_1,\dots,x_n)=(a_1 t,\dots, a_n t)$, then the smooth branch of $S_{r}^{n,n+1}$
is given by $(x_1,\dots,x_n,x_{n+1})=(a_1 t,\dots, a_n t, t^2)$. 
The curve is quasi-homogeneous. Its tangent cone is not the  $L_{r}^n$, but has 
an embedded point at the origin.

The invariants of this singularity extend the pattern for $d=2$.
One computes that $e=4r-3n-2$, $\dim T^1_\ell=0$ for $\ell>0$,
$\dim T^1_0=(n-1)(r-n-1)$ and $\dim T^1_{-1}=(r-2)n$, with $r=\binom{n+1}2+1$. 
We conclude that the general such curve is not smoothable for $n\geq 6$.
For these curves $\dim T^1_{-1}$ is large, but all these deformations should be obstructed.

If we want to add one more line, we have two choices. If we want to increase $\delta$ with
$2$, we need to increase the embedding dimension. To have a singularity better suited
to direct computation (for small $n$), we keep the embedding dimension constant, and increase
$\delta$ by three each time we add (in a certain range) a 
smooth branch in $\A^{n+1}$ of the same form $(a_1 t,\dots, a_n t, t^2)$.
Also these curves  $S_r^{n,n+1}$ are not smoothable.  

The curves $S_r^{n,n+1}$ are quasi-homogeneous, with variables 
$x_i$, $i=1,\dots,n$ of weight 1 and $x_{n+1}$ of weight 2. We expect
the ideal of $S_r^n$  for $\binom{n+1}2< r \leq \frac34 \binom{n+2}3$ to be  
generated by $\binom{n+2}3+n-r$ cubic equations and one equation of degree 4
containing the monomial $x_{n+1}^2$. 
The $\binom{n+2}3+n-r$ cubic equations alone generate the ideal of 
the union of $S_r^{n,n+1}$ and the $x_{n+1}$-axis.
The number of moduli for both curves is the same.
Explicit computations show that this is indeed true for $4\leq n \leq 7$.

\begin{defn}
The curve  $AS_r^{n,n+1}\subset \A^{n+1}$ is the curve with $r+1$ branches
consisting of the union of $S_r^{n,n+1}$, the curve with $r$ smooth branches tangent to 
$\A^n\subset \A^{n+1}$, with the $x_{n+1}$-axis.
\end{defn}

Explicit computations for $n=6$ show that the curves
$AS_r^{6,7}$ and $S_r^{6,7}$ have no deformations of negative degree 
 in the range $26\leq r \leq 47$. These curves do not deform into curves of a different type.

\subsection{Irreducible curves}
We consider Mumford's construction of non-smoothable curve singularities \cite{Mu}. 
Let $(t)\subset \k[t]$ be the ideal defining the point
$0\in \A^1$. 
Choose integers $1<n<d$ and let $V$ be a $\k$-vector space
with $(t^{2d})\subset V \subset (t^d)$ and $\dim V/(t^{2d})=n$. Then $\k+V$ is the
affine coordinate ring of a curve $C_V$ having a singular point with $\delta = 2d-n-1$.
These curves  are parametrised by a Grassmannian $G(n,d)$, but this 
family is too large too apply Proposition \ref{prop:greuel}, as it in general contains isomorphic curves.
We observe
that the general $C_V$ is a deformation of positive weight of the monomial
curve $C_{d,n}$ where $V/(t^{2d})$ has basis $(t^d,t^{d+1},\dots,t^{d+n-1})$, and that every
deformation of positive weight is of the form $C_V$. 

Moreover, the deformations
of positive weight are unobstructed, as we can perturb the parametrisation
arbitrarily with terms of higher weight.

\begin{thm}
The general equisingular deformation of the curve $C_{d,n}$ is not smoothable, 
if  $(n-6)(d-n-3)\geq 14$.
\end{thm}
\begin{proof}
As noted, the equisingularity stratum is  smooth with as tangent space
$T^1_{+}=\oplus_{\ell>0} T^1_{\ell}$. We compute $T^1_{\ell},$ using Buchweitz' formula in
Theorem 
\ref{thm:monomialT1}.  The semigroup of the curve $C_{d,n}$
is  $\langle d,d+1,\dots,d+n-1\rangle$ if $2n> d$ and
$\langle d,d+1,\dots,d+n-1, 2d+2n-1, \dots, 3d-1\rangle$  if
$2n<d+1$. Generators larger than the conductor do not contribute to $A_{\ell}$,
so $A_{\ell}= \{i\in\{1,\dots,n\}\mid d+i-1+\ell\notin\S\}$.
As  the degree of the equations is at least $2d$, there are also no positive dimensional $V_{\ell}$.
Therefore  $\dim T^1_{\ell}=|A_{\ell}| -1$  and $T^1_{\ell}=0$
for $\ell\geq d$.
 The condition defining $A_{\ell}$ is
$d+n-1 < d+i-1 +\ell <2d$ so $n<i+\ell<d+1$. For fixed $i$ there are 
$d-n$ values of $\ell$ where these conditions are satisfied, and those satisfy $1\leq \ell \leq d-1$.  
Therefore $\dim T^{1}_{+} = n(d-n)-(d-1) =(n-1)(d-n-1)$.

To compute $t$ we observe that the generators of $\omega$ are
$\frac{dt}{t^{2d}},\dots,\frac{dt}{t^{d+n+1}}$, so $t=d-n$
and $e=2\delta+t-1 = 5d-3n-3$. Therefore the general equisingular deformation
is not smoothable if $(n-1)(d-n-1)\geq  5d-3n-3$, so $(n-6)(d-n-1)\geq  2n+2=2(n-6)+14$.
\end{proof}

\begin{ex}
The  example of smallest embedding dimension is the curve $C_{17,9}$ of  genus $g=24$ with 
semigroup 
\[\langle 17,18,19,20,21,22,23,24, 25 \rangle\]

We cannot use Buchweitz' criterion to conclude that an irreducible curve with
$g=\delta=24$ and conductor $c=34$ is not smoothable, because $(t^{2c})\subset \Hom(\omega,\mO)$
and $\dim \mO/(t^{68}) = 44 < 48=2\delta$.

On the other hand, Buchweitz' monomial curve (Example \ref{ex:buchcurv})
with semigroup $\langle 13,\dots,18, 20,22,23 \rangle$
deforms to other irreducible curves with the same $\delta$.
In particular, the deformation of the parametrisation where only $x_{22}=t^{22}+st^{19}$
and $x_{23}=t^{23}+st^{21}$ are deformed, gives a curve with 
semigroup $\langle 13,14,\dots,21\rangle$. 
The general fibre of the deformation, say the one with $s=1$, can also be seen
as deformation of positive weight of the smoothable  monomial curve $C_{13,9}$
with semigroup $\langle 13,14,\dots,21\rangle$. Here we have that 
$(n-6)(d-n-3)=3 <14$, but the general equisingular deformation of the curve 
$C_{13,9}$ is not smoothable.
\end{ex}

\begin{remark}
The curve $C_{d,n}$ is always smoothable. For $2n\leq d+1$ the singularity is determinantal,
with equations coming from the matrix
\[
\begin{vmatrix}
x_d      & \dots & x_{d+n-2}  & x_{d+n-1}^2 & x_{2d+2n-1} & \dots & x_{3d-2}&   x_{3d-1} \\
x_{d+1} & \dots & x_{d+n-1}& x_{2d+2n-1} & x_{2d+2n} & \dots & x_{3d-1}  &  x_d^3
\end{vmatrix}
\]
For $2n> d+1$ the curve deforms into an $L_d^n$
which lies on the cone over a rational normal curve and is therefore smoothable.
\end{remark}

\begin{prop}\label{prop:equisinsmooth}
Every curve on the equisingularity stratum of $C_{n,d}$ deforms into curve singularity with only smooth
branches; if $d\leq 2n-1$ it deforms into $L_d^n$ and if $d\geq 2n$ it deforms into a curve 
$S_d^{n,N}$ with  $N=\max\{n,n+d-\binom{n+1}2\}$.
\end{prop}

\begin{proof}
In the case $d\geq 2n$  a curve equisingular with $C_{n,d}$  can be param\-etrised 
as
\[
\begin{cases}
x_i = t^i(1 + \varphi_i(t)), &  d\leq i \leq d+n-1 \\
x_j = t^j ,                & 2d+2n-1\leq j \leq 3d-1
\end{cases}
\]
where the polynomials $\varphi_i(t)$ contain only powers of $t$ in the range $[d+n-i,2d-1-i]$.
The deformation of the parametrisation
\[
\begin{cases}
x_i = (t^d-s)t^{i-d}(1 + \varphi_i(t)), &  d\leq i \leq d+n-1 \\
x_j = (t^d-s)^2t^{j-2d} ,                & 2d+2n-1\leq j \leq 3d-1
\end{cases}
\]
is $\delta$-constant and therefore flat. For $s\neq0$ the curve has smooth branches,
each of which has as tangent line a line through the origin in $\A^n\subset \A^{d-n+1}$.
The embedding dimension is  $\max\{n,n+d-\binom{n+1}2\}$.

For $d<2n$ the deformation $x_i = (t^d-s)t^{i-d}(1 + \varphi_i(t))$ for  $ d\leq i \leq d+n-1$
is a $\delta$-constant deformation into $L_d^n$.
\end{proof}


\section{Gorenstein curves}

\subsection{}
Buchweitz' criterion concerns deformations to curves with at most Gorenstein singularities
(see Corollary \ref{cor:rb-crit}), so it does not apply to symmetric semigroups. 
 Based on a observation of Stöhr, Torres gave a construction  \cite[Scholium 3.5]{To}
of Gorenstein non-Weierstrass semigroups.
Let $\widetilde\Gamma$ be a non-Weierstrass semigroup of genus $\gamma$ and
 let $g\geq 6\gamma +4$. Then the semigroup
 \[
 \Gamma=\{2n \mid n\in \widetilde\Gamma\}\cup \{2g-1-2t\mid t\in \Z\setminus \widetilde\Gamma\}
 \]
 is a symmetric non-Weierstrass semigroup. We refer to these as semigroups of Stöhr-Torres type.

\begin{ex}
The smallest example is 
the semigroup \(
\S=2\langle13,\dots,18,20,22,23\rangle + \langle149,151,157,161\rangle
\)
with $g=100$.
The embedding dimension is 13. There are 66 equations.
Many of them can be given by an incomplete determinantal, the remaining ones are
rolling factors. It is possible to compute the generators of $T^1$ as $\mO$-module.
The deformation of lowest weight is of rolling factors type and changes only the equations of
highest weight. We also can determine a deformation which covers the deformation
of the Buchweitz curve to $L_{13}^9$. Combining these one finds a deformation to a singularity 
with reduced tangent cone. Explicitly it is given by the following partial determinantal,
which we write transposed compared with the determinantal for the Buchweitz curve
in Example \ref{ex:buchcurv}, and rolling factors equations expressing the products of the variables
$\x(149), \dots, \x(161)$.
\[
\begin{vmatrix}
\x(26) &    \x(28) &     \x(30) &    \x(32) &     \x(34)  & \cdot \\
\x(28) &     \x(30) &    \x(32) &    \x(34) &     \x(36) & \x(40)\\
 \x(30) &    \x(32) &    \x(34) &    \x(36)  &    \cdot   &\cdot \\
 \x(32) &   \x(34) &    \x(36)  &   \cdot &     \x(40) & \x(44)\\
 \x(34) &    \x(36)  &   \cdot  &     \x(40)  &    \cdot & \x(46)\\
 \x(36)  &   \cdot &     \x(40)   &   \cdot &   \x(44)   & \cdot\\
 \x(40)   &   \cdot  &   \x(44)&     \x(46)   &   \cdot & \x(26)( \x(26)+\t(26))\\
 \x(44)&     \x(46)   &   \cdot  &     \cdot&  \x(26)( \x(26)+\t(26)) &\x(30)( \x(26)+\t(26))  \\
 \x(46)   & \cdot     & \cdot     & \x(26)( \x(26)+\t(26))  &  \x(28) ( \x(26)+\t(26))
 &\x(32)( \x(26)+\t(26))  \\
 \x(149)  & \x(151) & \cdot&   \cdot     & \x(157) & \x(161) \\
 \x(151) &   \cdot&   \cdot&  \x(157)&\cdot& \cdot\\
 \x(157)&\cdot& \x(161)&\cdot&   \cdot& \cdot
\end{vmatrix}
\]
\begin{gather*}
\x(149)^2-(\x(36) \x(28) \x(26)^7+\t(246) )\x(26)^2 , \\                                        
\x(151) \x(149)-(\x(36) \x(28) \x(26)^7+\t(246) ) \x(28) \x(26) , \\                               
\x(151)^2-(\x(36) \x(28) \x(26)^7+\t(246) )\x(30) \x(26) , \\                                    
\x(157) \x(149)-(\x(36) \x(28) \x(26)^7+\t(246) ) \x(30)^2 , \\                                   
\x(157) \x(151)-(\x(36) \x(28) \x(26)^7+\t(246) ) \x(32) \x(30) , \\                               
\x(161) \x(149)-(\x(36) \x(28) \x(26)^7+\t(246) ) \x(34) \x(30) , \\                               
\x(161) \x(151)-(\x(36) \x(28) \x(26)^7+\t(246) ) \x(36) \x(30) , \\                               
\x(157)^2-(\x(36) \x(28) \x(26)^7+\t(246) )\x(40) \x(28) , \\
\x(161) \x(157)-(\x(36) \x(28) \x(26)^7+\t(246) ) \x(46) \x(26) , \\                         
\x(161)^2-(\x(36) \x(28) \x(26)^7+\t(246) )\x(46) \x(30).  
\end{gather*} 
Retaining only the quadratic part of all these equations with $\t(26)=\t(246)=1$ 
gives a homogeneous singularity, which describes a $L_{26}^{13}$. 
We rename the variables: $\x(26+2k)$ becomes $\x(k)$ 
and $\x(149+2k)$ becomes $\y(k)$. The singularity  is a double cover of the non-smoothable
$L_{13}^9$ into which Buchweitz' curve deforms, see Example \ref{ex:buchcurv}.
This curve lies on the projection of the cone over the rational normal curve of degree $10$
onto $\A^9$: writing $\x(k)=u^{10-k}v^k$ describes the lines as $u^{13}=v^{13}$.
The double cover is then given by $\y(6)=\pm u^4v^6$, $\y(k)=\pm \x(k)$, with the same signs. 
This singularity  $L_{26}^{13}$ has no non-trivial deformations of positive weight, so 
the fibre of the deformation over $\t(26)=\t(246)=1$ is isomorphic to it.
\end{ex}

\begin{prop}\label{prop:goren}
The $L_{26}^{13}$ described above is a non-smoothable Gorenstein curve singularity.
\end{prop}
\begin{proof}
A computer computation shows that $T^1$ is concentrated in degree 0, of dimension $78$, the dimension of the moduli space.
\end{proof}

The fact  that this deformation of lowest degree 
leads to a non-smoothable singularity, suggests that 
the singularity is not smoothable at all. More evidence is given by a computation of 
the generators of $T^1$ as $\mO$-module, which  shows that
all the perturbations of the equations lie in the square of the maximal ideal. 
We expect this to hold in general.

\begin{conj} 
The Gorenstein monomial curves of Stöhr-Torres type are not smoothable. 
\end{conj}

\subsection{Self-associated point sets}
The explicit Gorenstein curve $L_{26}^{13}$ of Proposition \ref{prop:goren}
is a cone over a well-known type of
 point set: self-associated point sets \cite{Co}. 

\begin{defn} 
A set $\Gamma$ of $2n$ ordered points in $\P^{n-1}$ is self-associated 
if its Gale transform is projectively equivalent to $\Gamma$.
\end{defn}

We consider in general   cones $L_{2n}^n$  over   self-associated point sets.
We assume that every subset of $2n-1$ points is in uniform position (this is not the case
for the $L_{26}^{13}$ above). In particular
 the $2n-1$ points pose independent conditions on quadrics.
If the $2n$ points are self-associated, then they  fail  by one to impose
independent conditions on quadrics and the $L_{2n}^n$
is Gorenstein, see \cite[Sect. 7]{EP}.
The configuration is characterised
by the fact that quadrics passing through any $2n-1$ of the points pass through remaining point.
The number of moduli is $n(n-1)/2$; this was classically known, see \cite[Cor. 8.4]{EP}. 
Ordered sets of $2n$  self-associated points can be parametrised as $(I,P)$ with 
$P\in \operatorname{SO}(n,\k)$. For computations it is more convenient to parametrise 
$ \operatorname{SO}(n,\k)$ using the Cayley transform $A\mapsto (I+A)^{-1}(I-A)$ on the
set of matrices with $I+A$ invertible; it induces a birational map between skew-symmetric and special orthogonal matrices. A general self-dual configuration has a skew normal form
$(I+S,I-S)$ with $S$ skew-symmetric \cite[Thm. 2.9]{BK}.

\begin{thm}\label{thm:saccnotsm}
The general Gorenstein $L_{2n}^n$ is not smoothable if $n>9$.
\end{thm}

\begin{proof}
Let $C$ be a Gorenstein $L_{2n}^n$.  Then the points fail
 to impose independent conditions on quadrics but every
maximal proper subset imposes independent conditions on quadrics \cite[Sect. 7]{EP}. We
may also assume that every subset of $n$ points span $\P^{n-1}$.
Therefore a  subcurve $C_{2n-1}=L_{2n-1}^n$ has $\delta=3n-3$ and as all quadrics vanishing on
$C_{2n-1}$ also vanish on the remaining line $L$ we have that $L\cdot C_{2n-1}=3$.
So $\delta(C)=3n$ and the Deligne number is $e=\mu = 
\delta+(\delta-r+1)=3n+(n+1)=4n+1$.
As  the family of Gorenstein  $L_{2n}^n$ has dimension $n(n-1)/2$, the general such curve
is not smoothable if $n(n-1)/2\geq 4n+1$.
\end{proof}

We can also find this conclusion in another way. 
Petrakiev observed that for large $n$ the general self-associated point set is not  a hyperplane section
of a canonical curve \cite[Sect. 4]{Pe}.
The number of moduli of  hyperplane sections of  canonical curves
of genus $g$ in $P^{g-1}$
 is $(3g-3) + (g-1)$. Therefore the general $L_{2g-2}^{g-1}$ is not negatively
smoothable if $(g-1)(g-2)/2>4g-4$, which is the same bound as the Theorem above. 
In fact, for such curves negatively smoothable and smoothable are the same.

\begin{prop}
A Gorenstein $L_{2n}^n$ is always negatively graded.
\end{prop}

 \begin{proof}
 We compute in the same way as in Proposition \ref{prop:tplus}.
 We have $K_{\ell}=\mO_{\ell}$ for $\ell\geq 3$. Therefore $(\ker 
 {\partial f})_{\ell}=
 (K/\mO)^n_{\ell}$ for $\ell\geq 1$.  This shows that $T^1_{\ell}=0$ for $\ell>1$. For 
 $\ell=1$ we have $\dim (K/\mO)_1=\dim K_2/\mO_2=1$.
  Because $\dim K_1 = 2n$, in order to show that $T^1_1=0$ 
 we have to show that  $\dim \{ a\in K_1\mid a\m_1\subset \mO_2\}=n$.  
 An element $a\in K_1$
 has the form $(a_1t_1,\dots,a_{2n}t_{2n})$. By subtracting elements in $\m_1\subset
 \{ a\in K_1\mid a\m_1\subset \mO_2\}$,
  we can achieve
 that $a_1=\dots=a_n=0$. There exist a linear form $l\in\m_1$ such that $l$ vanishes
 in the points $P_{n+1}$, \dots, $P_{n+j-1}$, $P_{n+j+1}$, \dots, $P_{2n}$, but not in $P_{n+j}$,
 for $1\leq j \leq n$. Then $l$
 has as element of $K_2$ the form $(0,\dots,0,a_{n+j}l(P_{n+j})t_{n+j}^2,0\dots,0)$ and
 this is not an element of $\mO_2$ if $a_{n+j}\neq0$. Therefore 
$ \{ a\in K_1\mid a\m_1\subset \mO_2\}=\m_1$ and has dimension $n$.
 \end{proof}
 
 \begin{thm}
 The general Gorenstein $L_{2g-2}^{g-1}$ is smoothable if $g\leq 8$ or $g=10$ and not smoothable otherwise.
 \end{thm}
 \begin{proof}
A  Gorenstein $L_{2g-2}^{g-1}$ is negatively smoothable if the corresponding point set $\Gamma$ is  a
hyperplane section of a canonically embedded curve. This is classically known for
$g\leq 6$. The cases $g=7$ and $g=8$ are the main results of \cite{Pe}.  The general canonical
curve of genus 7 and 8 is a linear section of a Mukai Grassmannian, and Petrakiev shows that 
the same holds for the general $\Gamma$. For $g=9$ the general canonical curve
is a codimension 5 linear section of the Lagrangian Grassmannian 
$\operatorname{LG}(3,6)\subset \P^{14}$. The number of  moduli of codimension 6 linear
sections is 27 \cite[Sect. 4]{Pe}, whereas the moduli space of $\Gamma$ has dimension 28.
The dimension of a smoothing component of $L_{16}^8$ is 33, and the cone over
$\operatorname{LG}(3,6)$
is the total space of the versal deformation in negative degrees, so the smoothing component
intersects the equisingularity stratum in a $33-6=27$-dimensional space.  Therefore the general 
$L_{16}^8$  is not smoothable.
 
 For $g=10$ the general canonical curve does not lie on a $K3$ surface \cite{Muk}.
 A direct computation for a random $L_{18}^9$ (constructed from the skew normal form)
 shows that $\dim T^1_{-1}=1$ and that this infinitesimal
 deformation can be extended to a deformation, whose total space is the cone over a 
 canonically embedded curve. 
 
 For $g\geq 11$ the result follows from Theorem \ref{thm:saccnotsm}.
  \end{proof}

\section{Generic curves}
The first occurrence of the  term generic singularity  seems to be in a famous paper by 
Schlessinger \cite{Sc}. He says that a singularity is ``generic'' if it is not the specialisation of any other
singularity $X'$, where a specialisation is defined as 1-parameter deformation
with $X$ the special fibre and the other fibres all isomorphic to $X'$. Under such a definition
the curve  consisting of  an ordinary cusp and $11$ lines in general
position through the origin in    $\A^8$ of Proposition \ref{prop:cusplines} is ``generic''.

Iarrobino \cite{Ia2,IE} defines the term for zero-dimensional 
singularities using the Hilbert scheme. A generic singularity is one parametrised by a
generic point of a component of the Hilbert scheme parametrising only irreducible schemes.
To give explicit examples Iarrobino and Emsalem \cite{IE} look at almost-generic thick points,
meaning that  such a point  deforms only to other thick points of the same type (a notion they 
deliberately leave vague) and the parametrising point lies on a
single component of the Hilbert scheme.  They use the term ``generic'' for such singularities.

This point of view suggests to define a generic (curve) singularity as one parame\-trised by a
generic point  of a component of a base space of a versal deformation, excluding
smooth points.  A singularity is ``generic'' if its  base space has only one component,
and it has only equisingular deformations (for space curves not a precise concept either, see \cite{BG}).
A general non-smoothable homogeneous $L_n^r$ having deformations of positive weight is 
``generic'', while the generic singularity with the same tangent cone is not quasi-homogeneous. The curve of Proposition \ref{prop:cusplines} is not ``generic'' in this sense.

All our examples of non-smoothable curves are basically based
on curves with smooth branches or on monomial curves. This is not a severe restriction, as monomial 
curves are in a certain sense the most singular ones. The ones we encountered deform into
curves with smooth branches, see e.g. Proposition \ref{prop:equisinsmooth}. We did not
find irreducible curves deforming into $L_r^n$ with large $d(n,r)$. This shows how limited our 
 knowledge is. Nevertheless, based on our examples we offer:

\begin{conj}
All branches of generic curve singularities are smooth.
\end{conj}

\section*{Acknowledgement}
We express our gratitude to the referee for all suggestions to improve the paper.

\end{document}